\documentclass{amsart}
\usepackage{geometry,color}                
\usepackage{amssymb}

\title{Nonlinear BL}

\newcommand{\R}{\mathbb{R}}

\newcommand{\BL}{\operatorname{BL}}

\newcommand{\linearmapj}{L_j}
\newcommand{\linearfamily}{\mathbf{L}}

\newcommand\bl{{\operatorname{BL}}}

\newtheorem{theorem}{Theorem}[section]
\newtheorem{lemma}[theorem]{Lemma}
\newtheorem{prop}[theorem]{Proposition}

\newtheorem{corollary}[theorem]{Corollary}
\newtheorem{definition}[theorem]{Definition}

\theoremstyle{remark}
\newtheorem*{remark*}{Remark}

\newtheorem*{question*}{Question}
\newtheorem*{notation}{Notation}

\title[The Nonlinear Brascamp--Lieb inequality for simple data]{The Nonlinear Brascamp--Lieb inequality for simple data}
\author{Jonathan Bennett, Neal Bez, Stefan Buschenhenke and Taryn C. Flock}
\thanks{This work was supported by the European Research Council [grant
number 307617] (Bennett, Buschenhenke, Flock), JSPS Grant-in-Aid for Young Scientists A [grant number 16H05995] and JSPS Grant-in-Aid for Challenging Exploratory Research [grant number 16K13771-01] (Bez)}
\address{Jonathan Bennett and Stefan Buschenhenke: School of Mathematics, The Watson Building, University of Birmingham, Edgbaston,
Birmingham, B15 2TT, England.}
\email{J.Bennett@bham.ac.uk; S.Buschenhenke@bham.ac.uk}
\address{Neal Bez: Department of Mathematics, Graduate School of Science and Engineering,
Saitama University, Saitama 338-8570, Japan.}
\email{nealbez@mail.saitama-u.ac.jp}
\address{Taryn C. Flock: Department of Mathematics and Statistics,
 Lederle Graduate Research Tower,
 University of Massachusetts,
 710 N. Pleasant Street,
 Amherst, MA 01003-9305, USA.}
\email{flock@math.umass.edu}

\begin{document}
\begin{abstract}
We establish a nonlinear generalisation of the classical Brascamp--Lieb inequality in the case where the Lebesgue exponents lie in the interior of the finiteness polytope. As a corollary we show that the best constant in Young's convolution inequality in a small neighbourhood of the identity of a general Lie group, approaches the euclidean constant as the size of the neighbourhood approaches zero, answering a question of Cowling, Martini, M\"uller and Parcet. Our proof consists of running an efficient, or ``tight", induction on scales argument which uses the existence of gaussian extremisers to the underlying linear Brascamp--Lieb inequality in a fundamental way.
\end{abstract}
\maketitle
\section{Introduction}
The classical Brascamp--Lieb inequality is a far-reaching common generalisation of many well-known inequalities in analysis, including the multilinear H\"older, Young convolution and Loomis--Whitney inequalities. It takes the form
\begin{equation}\label{BL}
\int_{\mathbb{R}^n}\prod_{j=1}^m (f_j\circ L_j)^{p_j}\leq C\prod_{j=1}^m\left(\int_{\mathbb{R}^{n_j}}f_j\right)^{p_j},
\end{equation}
where for each $1\leq j\leq m$, the map $L_j:\mathbb{R}^n\rightarrow\mathbb{R}^{n_j}$ is a linear surjection (thus $n_j\leq n$), the exponent $p_j\in [0,1]$, the function $f_j$ is nonnegative and integrable, and all integration is with respect to Lebesgue measure. We refer to $m$ as the \emph{level of multilinearity} of the Brascamp--Lieb inequality since \eqref{BL} is manifestly equivalent to the multilinear form bound
\begin{equation}\label{BLform}
\Bigl|\int_{\mathbb{R}^n}\prod_{j=1}^m f_j\circ L_j\Bigr|\leq C\prod_{j=1}^m\|f_j\|_{L^{q_j}(\mathbb{R}^{n_j})}
\end{equation}
where $q_j=1/p_j\in [1,\infty]$. As is by now quite standard, we denote by $\linearfamily$ and $\mathbf{p}$ the $m$-tuples $(\linearmapj)_{j=1}^m$ and $(p_j)_{j=1}^m$ respectively, and refer to  $(\linearfamily,\mathbf{p})$ as the \textit{Brascamp--Lieb datum}. We denote by $\mbox{BL}(\linearfamily,\mathbf{p})$ the smallest constant $C$ in \eqref{BL} over all input functions $\mathbf{f}=(f_j)_{j=1}^m$, and refer to it as the \textit{Brascamp--Lieb constant}. Of course $\mbox{BL}(\linearfamily,\mathbf{p})$ may be infinite, as is easily seen to be the case, for instance, if $$
\bigcap_{j=1}^m \ker L_j
$$
is nontrivial, or the scaling condition
\begin{equation}\label{scal}
n=\sum_{j=1}^m p_jn_j
\end{equation}
is violated.

It was shown in \cite{BCCT1} (see \cite{Ba2} and \cite{CLL} in the case of rank-$1$ maps $L_j$) that $\mbox{BL}(\linearfamily,\mathbf{p})<\infty$ if and only if both the scaling condition
\eqref{scal} and
$$
\dim (V)\leq\sum_{j=1}^m p_j\dim(L_jV)
$$
hold for all subspaces $V$ of $\mathbb{R}^n$. Since all dimensions here are integer-valued, this finiteness characterisation consists of finitely many linear inequalities in the exponents $p_j$. This gives rise to the \emph{finiteness polytope} $$\Pi(\mathbf{L})=\{\mathbf{p}:\mbox{BL}(\linearfamily,\mathbf{p})<\infty\}$$ lying inside the scaling hyperplane \eqref{scal}. Following \cite{BCCT1}, we refer to a Brascamp--Lieb datum $(\linearfamily,\mathbf{p})$ as \emph{simple} if
$\mathbf{p}$ lies on the interior of $\Pi(\mathbf{L})$.

A well-known example is the classical sharp Young convolution inequality of Beckner \cite{Beckner}, \cite{Beckner2} and Brascamp--Lieb \cite{BL}, which may be stated as
\begin{equation}\label{Young}
\int_{\mathbb{R}^d\times\mathbb{R}^d}f_1(y)^{p_1}f_2(x-y)^{p_2}f_3(x)^{p_3}dydx\leq (C_{p_1}C_{p_2}C_{p_3})^d\Bigl(\int_{\mathbb{R}^d}f_1\Bigr)^{p_1}\Bigl(\int_{\mathbb{R}^d}f_2\Bigr)^{p_2}\Bigl(\int_{\mathbb{R}^d}f_3\Bigr)^{p_3},
\end{equation}
where $p_1,p_2,p_3\in [0,1]$, ${p_1}+{p_2}+{p_3}=2$, and $C_r=(1-r)^{1-r}/r^r$. In this case the extreme points of the finiteness polytope $\mathbf{p}=(1,1,0), (1, 0, 1), (0,1,1)$, are elementary consequences of Fubini's theorem. In the case of simple data, which arises when $p_1,p_2,p_3\in (0,1)$, the constant $(C_{p_1}C_{p_2}C_{p_3})^d<1$, and is attained by suitably scaled isotropic centred gaussian inputs $\mathbf{f}$.

In \cite{L} Lieb showed that $\mbox{BL}(\linearfamily,\mathbf{p})$ is exhausted by centred gaussian inputs; that is, functions of the form
\begin{equation}\label{gau}
f_j(x)=c_j\exp(-\pi\langle A_jx,x\rangle),
\end{equation}
where $A_j$ is a positive-definite $n_j\times n_j$ matrix, and $c_j$ a positive real number. A full characterisation of the data $(\linearfamily,\mathbf{p})$ for which the Brascamp--Lieb inequality \eqref{BL} has \emph{gaussian extremisers} appeared in \cite{BCCT1}, following the case of rank-1 maps $L_j$ in \cite{Ba2} (see also \cite{CLL}).
In particular, it was shown in \cite{BCCT1} that if a datum $(\linearfamily,\mathbf{p})$ is simple, then there are unique gaussian extremisers, up to the natural isotropic scale-invariance of the inequality.\footnote{By the scale invariance of \eqref{BL} (under the necessary condition \eqref{scal}), it is easily seen that if $\mathbf{f}=(f_j)_{j=1}^m$ is an extremiser, then $(c_jf_j(\lambda\cdot))_{j=1}^m$ is an extremiser for all choices of $c_1,\hdots, c_m,\lambda>0$. This scale-invariance will be important to us later.} This feature of simple Brascamp--Lieb data will play an important role in this paper.

Recently, beginning with \cite{BCW}, certain \emph{nonlinear} generalisations of the Brascamp--Lieb inequality have come to the fore in harmonic analysis and dispersive PDE. The idea is to replace the linear surjections $L_j$ by smooth submersions $B_j$ defined in a neighbourhood of a point $x_0\in\mathbb{R}^n$, and ask in what sense the inequality can persist. In what follows we write $\mathbf{B}=(B_j)_{j=1}^m$ and $\mbox{d}\mathbf{B}(x_0)=(dB_j(x_0))_{j=1}^m$ for $x_0\in\mathbb{R}^n$. Our main theorem is the following:

\begin{theorem}\label{mainthm} Let $\varepsilon>0$.
If the datum $(\linearfamily,\mathbf{p})$ is simple, and $\mathbf{B}$ is a family of $C^2$ submersions in a neighbourhood of a point $x_0\in\mathbb{R}^n$ for which $dB_j(x_0)=L_j$ for each $1\leq j\leq m$, then there exists a neighbourhood $U$ of $x_0$ for which
$$
\int_U\prod_{j=1}^m(f_j\circ B_j)^{p_j}\leq (1+\varepsilon)\bl(\linearfamily,\mathbf{p})\prod_{j=1}^m\left(\int_{\mathbb{R}^{n_j}}f_j\right)^{p_j}.
$$
\end{theorem}
We conclude this introduction with a number of remarks on the context and scope of Theorem \ref{mainthm}.

Theorem \ref{mainthm} is rather different from existing nonlinear Brascamp--Lieb inequalities, which have in all cases required some quite rigid structural conditions either on the nonlinear maps $B_j$ or  on the kernels of the underlying linear maps $L_j$. These conditions on the kernels\footnote{with the exception of some unpublished examples of Betsy Stovall} essentially amount to the basis condition
\begin{equation}\label{basis}
\bigoplus_{j=1}^m \ker L_j=\mathbb{R}^n,
\end{equation}
which keeps the complexity at the level of the classical Loomis--Whitney inequality; see \cite{BCW}, \cite{BB}, \cite{BHT}, \cite{KS}, \cite{C},
and an application to dispersive PDE in \cite{BHHT}, \cite{BH}.
However, if one is prepared to accept an $\varepsilon$-loss in the scale of Sobolev spaces, an \emph{arbitrary} Brascamp--Lieb datum admits nonlinear perturbations -- see \cite{BBFL}, where it is shown that whenever $\mbox{BL}(\mbox{d}\mathbf{B}(x_0)),\mathbf{p})<\infty$, there exists a neighbourhood $U$ of $x_0$ such that for each $\varepsilon>0$,
\begin{equation}\label{NLBLform}
\Bigl|\int_{U}\prod_{j=1}^m f_j\circ B_j\Bigr|\leq C_\varepsilon\prod_{j=1}^m\|f_j\|_{L^{q_j}_\varepsilon(\mathbb{R}^{n_j})}.
\end{equation}
Here $(\mbox{d}\mathbf{B}(x_0))=(dB_j(x_0))_{j=1}^m$ and $q_j=1/p_j\in [1,\infty]$. A very different argument of Zhang \cite{Z} (building on work of Guth \cite{Guth}, and Bourgain and Guth \cite{BG}; see also Carbery and Valdimarsson \cite{CV}) allows for \emph{polynomial} $\mathbf{B}$, but generates a bound that depends strongly on its degree.

Theorem \ref{mainthm} was previously unknown with \textit{any} finite constant in place of the $1+\varepsilon$ factor, and we do not know how to prove such a (weaker) qualitative statement without the full gaussian-based argument of this paper. We note that Carbery's recent factorisation-based proof of the nonlinear Loomis--Whitney inequality \cite{C} is also well-suited to obtaining near-optimal constants.

It is not possible to take $\varepsilon<0$ in the statement of Theorem \ref{mainthm} for \emph{any} datum $(\mathbf{L},\mathbf{p})$, as may be seen by a slight modification of the argument in Lemma 6 from \cite{BBG}; see also \cite{CMMP}. (This argument is related to the so-called ``transplantation" method -- see for example \cite{KST} and \cite{Alessio}.) As may be expected, it is not in general possible to take $\varepsilon=0$ either, as $\mbox{BL}(\mbox{d}\mathbf{B}(x),\mathbf{p})$ varies with $x$ in all but very special situations.

The simplicity hypothesis in the statement of Theorem \ref{mainthm} is natural in perturbative contexts such as this since
$
\{\mathbf{L}:(\mathbf{L},\mathbf{p})\mbox{ simple}\}
$
is an open set. This was proved by Valdimarsson in \cite{ValdBestBest}, who showed further that the Brascamp--Lieb constant $\mbox{BL}(\cdot,\mathbf{p})$ is smooth near simple data points (see also \cite{BBFL}, \cite{BBCF} and \cite{GGOW}, where the regularity of the Brascamp--Lieb constant in $\mathbf{L}$ is studied at \emph{all} data points). As one might expect, this regularity will be important in our proof of Theorem \ref{mainthm}.

Theorem \ref{mainthm} is equivalent to a superficially stronger statement where the underlying euclidean spaces $\mathbb{R}^n$, $\mathbb{R}^{n_1},\hdots,\mathbb{R}^{n_m}$ (equipped with Lebesgue measures) are replaced by general smooth manifolds $M$, $M_1,\hdots,M_m$ (equipped with smooth densities). This is a reflection of the (essential) diffeomorphism-invariance of its statement. While this is arguably the natural context for such nonlinear Brascamp--Lieb inequalities, we choose to work in the euclidean setting for the sake of concreteness.

As an application of Theorem \ref{mainthm} we provide a positive answer to a recent question of Cowling, Martini, M\"uller and Parcet \cite{CMMP} relating to the best constant in the Young convolution inequality on a general Lie group, for input functions supported close to the identity element.
\begin{corollary}\label{cor} Suppose that $p_1,p_2,p_3\in (0,1)$ are such that $p_1+p_2+p_3=2$, and that $\varepsilon>0$.
If $G$ is a Lie group with identity $e$ and topological dimension $d$, then there exists a neighbourhood $U$ of $e$ for which
\begin{equation}\label{lsy}
\int_{G\times G}f_1(y)^{p_1}f_2(y^{-1}x)^{p_2}f_3(x)^{p_3}dydx\leq (1+\varepsilon)(C_{p_1}C_{p_2}C_{p_3})^d\Bigl(\int_{G}f_1\Bigr)^{p_1}\Bigl(\int_{G}f_2\Bigr)^{p_2}\Bigl(\int_{G}f_3\Bigr)^{p_3}
\end{equation}
for all $f_1, f_2, f_3$ supported in $U$.
\end{corollary}
Thus the best constant in an abstract Lie group setting approximates the euclidean constant in \eqref{Young} as the supports of the input functions $f_1, f_2, f_3$ are confined to ever smaller neigbourhoods of $e\in G$.
It is not difficult to see that Corollary \ref{cor} follows from Theorem \ref{mainthm}, as the differentials of the underlying maps $(x,y)\mapsto y$, $(x,y)\mapsto x$, $(x,y)\mapsto x^{-1}y$ at the identity correspond to the linear maps in the classical Young convolution inequality via the Baker--Campbell--Hausdorff formula.
As usual in this context, the measure appearing in the statement above is any Haar measure, be it left or right-invariant. We refer to \cite{CMMP}, \cite{KR} and \cite{Fourn} for further context and related results.

Finally, the nonlinear Brascamp--Lieb inequality is one of several generalisations of \eqref{BL} that have emerged in the last decade or so. In particular, certain \emph{combinatorial Brascamp--Lieb inequalities} of Kakeya type, and closely-related \emph{Fourier-analytic Brascamp--Lieb inequalities} related to the Fourier restriction conjecture, have been established and applied in a variety of contexts -- see for example \cite{BBFL}, \cite{BGD}, \cite{Z} and the references there.

\subsubsection*{Acknowledgments} The authors wish to thank Tony Carbery for many in-depth discussions on the subject of this paper. They also thank Michael Cowling and Alessio Martini for discussions surrounding the local sharp Young's convolution inequality \eqref{lsy}.

\section{The proof strategy}
Our proof of Theorem \ref{mainthm} builds on a strategy used by the first two authors in \cite{BB}, which was in turn inspired by work of Ball \cite{Ball} and Bejenaru, Herr and Tataru \cite{BHT}. The starting point is an inequality originating in \cite{Ball}, which captures an important self-similarity property of the Brascamp--Lieb inequality. As we shall now see, this property strongly suggests approaching any sort of ``perturbation" of the Brascamp--Lieb inequality by the method of induction-on-scales.\footnote{This has been fruitful on several occasions -- see \cite{BEsc}, \cite{BCT}, \cite{BHT}, \cite{BB}, \cite{Guth2}, \cite{BBFL}, along with closely-related papers in Fourier restriction theory beginning with \cite{Bo}.} In order to describe this we define the \emph{Brascamp--Lieb functional} $\mbox{BL}(\linearfamily,\mathbf{p}; \cdot)$ by
$$
\mbox{BL}(\linearfamily,\mathbf{p}; \mathbf{f})=\frac{\int_{\mathbb{R}^n}\prod_{j=1}^m (f_j\circ L_j)^{p_j}}{\prod_{j=1}^m\Bigl(\int_{\mathbb{R}^{n_j}}f_j\Bigr)^{p_j}},
$$
so that
\begin{equation}\label{supfunct}
\mbox{BL}(\linearfamily,\mathbf{p})=\sup_{\mathbf{f}}\mbox{BL}(\linearfamily,\mathbf{p}; \mathbf{f}).
\end{equation}
We begin by observing that for two $L^1$-normalised inputs $\mathbf{f}$ and $\mathbf{g}$ (in the sense that $\int f_j=\int g_j=1$ for all $1\leq j\leq m$),
\begin{eqnarray}\label{Ballineq}
\begin{aligned}
\mbox{BL}(\linearfamily,\mathbf{p}; \mathbf{f})\mbox{BL}(\linearfamily,\mathbf{p}; \mathbf{g})&=\int\Bigl(\prod_{j=1}^m (f_j\circ L_j)^{p_j}\Bigr)*\Bigl(\prod_{k=1}^m (g_k\circ L_k)^{p_k}\Bigr)\\
&=\int\Biggl(\int \prod_{j=1}^m[f_j(L_j(y))g_j(L_j(x-y))]^{p_j}dy\Biggr)dx\\
&=\int\Biggl(\int \prod_{j=1}^m[f_j(L_j(y))g_j(L_j(x)-L_j(y)))]^{p_j}dy\Biggr)dx\\
&=\int\Biggl(\int \prod_{j=1}^m[h_j^x(L_j(y)]^{p_j}dy\Biggr)dx,
\end{aligned}
\end{eqnarray}
where $h_j^x(z)=f_j(z)g_j(L_j(x)-z)$ for each $j$. Writing $\mathbf{h}^x=(h_j^x)_{j=1}^m$ and $\mathbf{f}*\mathbf{g}=(f_j*g_j)_{j=1}^m$, it follows that
\begin{eqnarray*}
\begin{aligned}
\mbox{BL}(\linearfamily,\mathbf{p}; \mathbf{f})\mbox{BL}(\linearfamily,\mathbf{p}; \mathbf{g})&=
\int \mbox{BL}(\linearfamily,\mathbf{p}; \mathbf{h}^x)\prod_{j=1}^m\Bigl(\int_{\mathbb{R}^{n_j}}h_j^x\Bigr)^{p_j}dx\\
&=\int \mbox{BL}(\linearfamily,\mathbf{p}; \mathbf{h}^x)\prod_{j=1}^m(f_j*g_j(L_j(x)))^{p_j}dx\\
&\leq \sup_x \mbox{BL}(\linearfamily,\mathbf{p}; \mathbf{h}^x)\mbox{BL}(\linearfamily,\mathbf{p}; \mathbf{f}*\mathbf{g}).
\end{aligned}
\end{eqnarray*}
This we refer to as Ball's inequality, and it should be noticed that the $L^1$ normalisation hypotheses on $\mathbf{f}$ and $\mathbf{g}$ may be dropped by homogeneity considerations.

Now, if $\mathbf{g}$ is an \emph{extremiser} for \eqref{supfunct}, then we may conclude that
\begin{equation}\label{Ball1}
\mbox{BL}(\linearfamily,\mathbf{p}; \mathbf{f})\leq \mbox{BL}(\linearfamily,\mathbf{p}; \mathbf{f}*\mathbf{g})
\end{equation}
and
\begin{equation}\label{Ball2}
\mbox{BL}(\linearfamily,\mathbf{p}; \mathbf{f})\leq \sup_x \mbox{BL}(\linearfamily,\mathbf{p}; \mathbf{h}^x).
\end{equation}
The inequalities \eqref{Ball1} and \eqref{Ball2} contain a surprising amount of information. For example, \eqref{Ball1} tells us that the set of extremisers for \eqref{BL} is closed under convolution -- a fact that may be used (see \cite{BCCT1}) along with the central limit theorem to deduce the existence of gaussian extremisers (provided the set of extremisers is non-empty of course). In the presence of a gaussian extremiser, inequality \eqref{Ball1} may be used again to deduce that the Brascamp--Lieb functional is non-decreasing as the inputs evolve under a suitable heat flow; again see \cite{BCCT1}.

The closely-related inequality \eqref{Ball2}, while more complicated, will be more important for our purposes. Informally, if we happen to have an extremising input $\mathbf{g}$ which resembles a bump function with small support, then the function
$$
h_j^x(z)=f_j(z)g_j(L_j(x)-z)
$$
is simply the function $f_j$ localised to a small neigbourhood of the point $L_j(x)$. The inequality \eqref{Ball2} then tells us that the Brascamp--Lieb functional increases as the general input functions $f_j$ are localised in this way. Our approach to proving Theorem \ref{mainthm} is motivated by this, and amounts to proving a suitable nonlinear variant of \eqref{Ball2}. This approach is particularly natural, as an input function $f_j$ with ``shrinking support" will be increasingly unable to detect nonlinear structure in $B_j$, ultimately allowing us to reduce matters to a linear Brascamp--Lieb inequality. Proving such a nonlinear variant of \eqref{Ball2} presents several difficulties. Perhaps the most obvious difficulty is in finding a substitute for the explicit use of linearity in \eqref{Ballineq}. Another difficulty relates to the fact that a simple datum $(\mathbf{L}, \mathbf{p})$ only has gaussian extremisers, which needless to say, are not compactly supported.\footnote{In \cite{BB} it is observed that a datum satisfying \eqref{basis} (which is never simple) does have suitable compactly supported extremisers, allowing such a proof of a nonlinear version of \eqref{Ball2} in that case.} We overcome this by \emph{truncating} gaussian extremisers, forming \emph{$\delta$-near extremisers} for suitably small $\delta>0$ -- that is, inputs $\mathbf{g}$ for which
$$
\mbox{BL}(\linearfamily,\mathbf{p}; \mathbf{g})\geq (1-\delta)\mbox{BL}(\linearfamily,\mathbf{p}).
$$
The purpose of truncating is to create small compact support, although this of course comes at a cost: if $\mathbf{g}$ is a $\delta$-near extremiser then \eqref{Ball2} must be replaced with
$$
\mbox{BL}(\linearfamily,\mathbf{p}; \mathbf{f})\leq (1+O(\delta))\sup_x \mbox{BL}(\linearfamily,\mathbf{p}; \mathbf{h}^x).
$$
Provided $\delta$ is taken from a sufficiently lacunary sequence $(\delta_k)$ converging to zero, such factors of $1+O(\delta)$ may be tolerated, as upon iteration it ultimately leads to a convergent
product $\prod_k(1+O(\delta_k))$, which can made as close to $1$ as we please. As may be expected, our argument will require us to employ gaussian extremisers for the linear datum $(\mbox{d}\mathbf{B}(x), \mathbf{p})$ for each $x$ in a neighbourhood of the fixed point $x_0$. Thus many different gaussians will be in play.

As will become apparent, our proof of Theorem \ref{mainthm} actually applies to $m$-tuples $\mathbf{L}$ lying in the set
\begin{equation*}\mbox{E}=\{\mathbf{L}': \mbox{BL}(\mathbf{L}',\mathbf{p};\cdot) \mbox{ has a gaussian extremiser}\},\end{equation*} for which there exists a map sending $\mathbf{L}'\in \mbox{E}$ to a corresponding $m$-tuple of extremising positive definite matrices $\mathbf{A}=(A_j)_{j=1}^m$ (identified with gaussian extremisers via \eqref{gau}), which is continuous in some neighbourhood of $\mathbf{L}$.
The set of $m$-tuples $\mathbf{L}$ for which $(\mathbf{L},\mathbf{p})$ is simple is a proper subset of this, and we choose to state Theorem \eqref{mainthm} in that setting. It is worth noting that the set of $m$-tuples $\mathbf{L}$ satisfying the kernel condition \eqref{basis} is also of this type, and so we implicitly provide a further proof of the nonlinear Loomis--Whitney inequality of \cite{BCW}, \cite{BB}.

\subsubsection*{Organisation of the paper} In Section \ref{sect:settingup} we set up the inductive scheme, reducing matters to a certain recursive inequality (or ``inductive step") and a ``base case" -- the forthcoming Propositions \ref{prop:Csimple} and \ref{prop:basecase} respectively. We prove Proposition \ref{prop:basecase} in Section \ref{sect:settingup}. In Section \ref{sect:gaussians} we introduce the families of gaussians that we shall need, and establish some of their key properties. In Section \ref{sect:step} we prove Proposition \ref{prop:Csimple}.

\section{Setting up the induction}\label{sect:settingup}
By translation invariance we may assume that $x_0=0$ in the statement of Theorem \ref{mainthm}.
As discussed in the previous section, our proof of Theorem \ref{mainthm} will proceed by induction on the size of the supports of the input functions $\mathbf{f}$. 
In order for there to be a base case for this induction it will be convenient to suppose that $\mathbf{f}$ satisfies a certain auxiliary regularity condition, upon which our estimates will be essentially independent. The passage to general integrable $\mathbf{f}$ will consist of an elementary limiting argument.
\begin{definition}\label{localconst} Suppose $\kappa>1,\mu>0$ and $\Omega$ is a measurable subset of $\R^d$. A nonnegative function $f$ is \textit{$\kappa$-constant at scale $\mu$ on $\Omega$} if $f(x)\leq\kappa f(y)$ whenever $x\in\Omega$ and $y\in\mathbb{R}^d$ are such that $|x-y|\leq\mu$. We denote by $L^1(\Omega;\mu,\kappa)$ the subset of $L^1(\mathbb{R}^d)$ with this property.
\end{definition}
The reader will observe that the definition above lacks symmetry in $x$ and $y$: we stipulate that $x\in\Omega$, while $y$ may extend a distance $\mu$ from $\Omega$. The reason for this technicality will become clear shortly. At this stage it is worth noting that for any fixed $\kappa>1$, a nonnegative function $f\in L^1(\Omega)$ may be approximated almost everywhere by functions in $L^1(\Omega;\mu,\kappa)$, provided $\mu$ is taken sufficiently small. One way to see this is to observe that the $d$-dimensional Poisson kernel $P_t(x)$ is $\kappa$-constant at scale $\mu$ on the whole of $\R^{d}$ provided $\mu$ is small enough depending on $t>0$ and $\kappa>1$. By linearity, the convolution $f*P_t$ is easily seen to inherit this regularity property for any nonnegative $f\in L^1(\mathbb{R}^{d})$, and so the claimed almost everywhere approximation follows from the Lebesgue differentiation theorem.

We will allow $\kappa$ to vary in a controlled manner through the induction. This flexibility in $\kappa$ is required when considering products of such ``locally constant" functions. In particular, we will need to appeal to the elementary fact that if $f\in L^1(\Omega;\mu,\kappa)$ and $g\in L^1(\Omega;\mu,\lambda)$, then $fg\in L^1(\Omega;\mu,\kappa\lambda)$. Considerations of this type naturally arise when introducing partitions of unity, as we shall to pass between scales.

We now set up the induction. For each $0<\delta\ll1$ and $y\in\mathbb{R}^n$, let
\begin{equation}\label{dilnhd}
U_\delta(y)=\{x:|x-y|\leq\delta\log(1/\delta)\}.
\end{equation}
For $u\in \R^n$, $\delta>0$, $\mu> 0$ and $\kappa>1$, let $\mathcal{C}(u,\delta,\mu,\kappa)$ denote the best constant $C$ in the inequality
$$\int_{U_\delta(u)}\prod_{j=1}^m(f_j\circ B_j(y))^{p_j} dy\leq C\prod_{j=1}^m\left(\int_{\R^{n_j}}f_j\right)^{p_j}$$
over all  inputs\footnote{We clarify that we use $2U_\delta(u)$ to denote the ball with centre $u$ and radius $2\delta \log (1/\delta)$.} $f_j\in L^1(B_j(2U_\delta(u));\mu,\kappa)$. We think of $\mathcal{C}(u,\delta,\mu,\kappa)$ as a certain (regularised and localised) nonlinear Brascamp--Lieb constant; here we are of course suppressing the dependence on $\mathbf{B}$ and $\mathbf{p}$ in our notation.
The mysterious-looking logarithmic factor in the definition of $U_\delta(y)$ is included as our passage between scales in the parameter $\delta$ (see the statement of Proposition \ref{prop:Csimple} below) will involve localising using \emph{gaussians}, for which logarithmic truncations turn out to be large enough to capture their essential shape.
We also rely heavily on the sub-power growth rate of the logarithm to make these truncated gaussians fit into the framework of $\kappa$-constant functions introduced in Definition \ref{localconst}. The details may be found in Section \ref{sect:gaussians}, where it becomes apparent that any function which grows slower than a sufficient small positive power, but no slower than the square root of the logarithm, would suffice to prove Theorem \ref{mainthm} -- the logarithm is simply a convenient example.
The requirement that $f_j\in L^1(B_j(2U_\delta(u));\mu,\kappa)$ may also seem unusual as the integral on the left-hand-side does not see the part of $f_j$ supported outside of $B_j(U_\delta(u))$, whereas the right-hand-side does. This merely technical feature will be important for closing the induction.

In what follows let $\alpha,\beta$ be fixed parameters satisfying $\alpha>1$, $\beta>0$ and $\alpha+\beta<2$. For merely technical purposes, fix a further parameter $\beta<\beta' <2-\alpha$.\footnote{The reader may wish to observe that the $C^2$ regularity hypothesis on $\mathbf{B}$ may be weakened to $C^{1,\theta}$ for any $\theta>0$, upon which the condition $\alpha+\beta<2$ should be tightened to $\alpha+\beta<1+\theta$ and the relation $\beta<\beta' <2-\alpha$ of course replaced with $\beta<\beta' <1+\theta-\alpha$.}

\begin{notation}
We write $C \simeq 1$ to mean $C$ is equal to a positive finite constant which depends on at most $\mathbf{B}$, $\mathbf{p}$, and the fixed parameters $\alpha, \beta, \beta'$. For $\varepsilon > 0$, we write $C \simeq_\varepsilon 1$ to mean that $C$ is equal to such a constant which may also depend on $\varepsilon$.
Finally, we write $A \lesssim B$ and $B \gtrsim A$ to mean $A \leq CB$, where $C \simeq 1$, and $A \sim B$ means that $A \lesssim B$ and $A \gtrsim B$ both hold.
\end{notation}

If $\delta$ is below a certain threshold in terms of $\mu$, then each $f_j\in L^1(B_j(2U_\delta(u));\mu,\kappa)$ will (effectively) cease to distinguish between $B_j$ and its best affine approximation at $u$
$$L^u_j:=B_j(u)+dB_j(u)(\cdot-u.)$$
This allows us to reduce the estimation of $\mathcal{C}(u,\delta,\mu,\kappa)$ to an application of the \emph{linear} Brascamp--Lieb inequality, and provides us with an effective ``base case" for the inductive argument. This base case is contained in the first of our two propositions below. Before stating this, we note that by the $C^2$ regularity of the $B_j$, we have
\begin{equation} \label{taylor}
|B_j(x) - L^u_jx| \lesssim |x-u|^2
\end{equation}
for all $x, u \in U_{\nu}(0)$ (with appropriately small $\nu \simeq 1$), and we make frequent use of this fact.
\begin{prop}[Base case] \label{prop:basecase}
There exists $\nu \simeq 1$ such that the following holds: if $u \in U_{\nu}(0)$, $\kappa > 1$, and if $\delta \in (0,\nu)$ and $\mu > 0$ satisfy $\delta^{\alpha+\beta'}\leq\mu$, then
\[
\mathcal{C}(u,\delta,\mu,\kappa) \leq \kappa^\sigma \bl(\mathrm{d}\mathbf{B}(u), \mathbf{p}),
\]
where $\sigma = \sum_{j=1}^m p_j$.
\end{prop}
\begin{proof} Let $\nu\simeq 1$ be a small constant to be chosen below. We fix $u \in U_{\nu}(0)$, and assume $\delta \in (0,\nu)$ and $\mu>0$ satisfy $\delta^{\alpha+\beta'} \leq\mu$.

By choosing $\nu\simeq 1$ sufficiently small and applying \eqref{taylor},
\[
|B_j(y)-L^u_jy| \lesssim \delta^2 \left(\log \frac{1}{\delta} \right)^2
\]
for all $y\in U_\delta(u)$. From this, combined with $\alpha+\beta'<2$ and the assumption $\delta^{\alpha+\beta'} \leq\mu$, it follows that
$
|B_j(y)-L^u_jy| \leq \mu,
$
for a perhaps smaller choice of $\nu \simeq 1$. Since $f_j\in L^1(2B_j(U_\delta(u));\mu,\kappa)$, we have
$
f_j(B_j(y))\leq \kappa f_j(L^u_jy)
$
for all $y\in U_\delta(u)$, and so
\begin{align*}
\int_{U_{\delta}(u)}\prod_{j=1}^m f_j(B_j(y))^{p_j} dy & \leq \kappa^\sigma \int_{U_{\delta}(u)}\prod_{j=1}^m f_j(L^u_jy)^{p_j} dy \\
& \leq \kappa^\sigma \mbox{BL}(\mathrm{d} \textbf{B}(u), \textbf{\mbox{p}}) \prod_{j=1}^m\left(\int_{\R^{n_j}}f_j\right)^{p_j},
\end{align*}
as required. Here we have used the translation-invariance of the linear Brascamp--Lieb inequality in the $f_j$ to remove the translations in the above affine approximation.
\end{proof}

Our objective is to prove a suitable recursive inequality involving the function $\mathcal{C}$, which upon iteration will establish Theorem \ref{mainthm}. This inequality, which we now state, may be viewed as a certain nonlinear version of \eqref{Ball2}.
\begin{prop}[Recursive inequality] \label{prop:Csimple}
There exists $\nu \simeq 1$ such that the following holds: if $u \in U_{\nu}(0)$, $\kappa > 1$, and if $\delta \in (0,\nu)$ and $\mu > 0$ satisfy $\delta^{\alpha+\beta'}>\mu$, then
\begin{equation}\label{inductionstep}
\mathcal{C}(u,\delta,\mu,\kappa) \leq (1 + \delta^{\beta}) \max_{x\in 2U_\delta(u)} \mathcal{C}(x,\delta^\alpha,\mu,\kappa\exp(\delta^\beta)).
\end{equation}
\end{prop}
The reader may wish to view \eqref{inductionstep} as a certain ``near-monotonicity property" of the function $\mathcal{C}$ in the parameter $\delta$, since the multiplicative factors $(1 + \delta^{\beta})$ and $\exp(\delta^\beta)$ are close to $1$ for small $\delta$. As we have mentioned, this is a certain approximate version of the exact monotonicity property \eqref{Ball2}. It is interesting to compare \eqref{inductionstep} with the related near-monotonicity results in \cite{BCT}, \cite{BHT}, \cite{BB} and \cite{TaoBil}.

We conclude this section by showing how Theorem \ref{mainthm} follows from Propositions \ref{prop:Csimple} and \ref{prop:basecase}.

\begin{proof}[Proof of Theorem \ref{mainthm}]
Fix $\varepsilon>0$ and set $\kappa=1+\varepsilon$.
By an elementary density argument, which we now sketch, it will be enough to show that there exists $\delta_0 \simeq_\varepsilon 1$ and $C \simeq 1$ such that
\begin{equation}\label{nuf}
\mathcal{C}(0,\delta_0,\mu, 1+\varepsilon)\leq (1+C\varepsilon)\mbox{BL}(\mathrm{d} \textbf{\mbox{B}}(0), \textbf{\mbox{p}})
\end{equation}
uniformly in $\mu>0$.
As we observed above, for each $t>0$, and $f_j\in L^1(2B_j(U_{\delta_0}(0)))$, the Poisson extension satisfies $f_j*P^{(n_j)}_t\in L^1(\mathbb{R}^{n_j}; \mu, \kappa)\subseteq L^1(2B_j(U_{\delta_0}(0));\mu, \kappa)$ provided $\mu$ is chosen sufficiently small. Hence \eqref{nuf}, combined with the $L^1$ normalisation of the $n_j$-dimensional Poisson kernel $P_t^{(n_j)}$, imply that
$$
\int_{U_{\delta_0}(0)}\prod_{j=1}^m (f_j*P^{(n_j)}_t\circ B_j)^{p_j}\leq (1+C\varepsilon)\mbox{BL}(\mathrm{d} \textbf{\mbox{B}}(0), \textbf{\mbox{p}})\prod_{j=1}^m\left(\int_{\mathbb{R}^{n_j}}f_j\right)^{p_j}
$$
uniformly in $t>0$. Since the nonnegative function $f_j*P^{(n_j)}_t$ converges to $f_j$ almost everywhere as $t\rightarrow 0$, Theorem \ref{mainthm} follows by Fatou's lemma. We now turn to the proof of \eqref{nuf}.

We fix a sufficiently small $\delta_0 \simeq_{\varepsilon} 1$ to be chosen to satisfy a number of constraints. The first of these is to take $\delta_0$ to satisfy $\delta_0\leq \nu/3$, where $\nu$ is provided by both Propositions \ref{prop:basecase} and \ref{prop:Csimple} (that we may take such a common $\nu$ follows from the fact that both propositions implicitly require that $\nu$ is sufficiently small).

Given $\mu>0$, either the base case provided by Proposition \ref{prop:basecase} holds directly, $\delta_0^{\alpha+\beta'}\leq \mu$, or $\mu$ satisfies the threshold condition, $\delta_0^{\alpha+\beta'}>\mu$. In the former case there is nothing more to prove. In the latter we shall use the recursive inequality \eqref{inductionstep} of Proposition \ref{prop:Csimple} to connect $\mathcal{C}(0,\delta_0,\mu, 1+\varepsilon)$ with the base case.

Defining the sequence $\delta_k=\delta_0^{\alpha^{k}}$, Proposition \ref{prop:Csimple} yields
\begin{equation}\label{toiterate}
\mathcal{C}(u,\delta_k,\mu,\kappa) \leq (1 + \delta_k^{\beta}) \max_{x\in 2U_{\delta_k}(u)} \;\mathcal{C}(x,\delta_{k+1},\mu,\kappa\exp(\delta_k^{\beta}))
\end{equation}
for all $u$ belonging to the algebraic sum
$
\widetilde{U}_k:=2U_{\delta_0}(0)+2U_{\delta_1}(0)+\cdots+2U_{\delta_{k-1}}(0)
$
and
$k$ such that $\delta_k^{\alpha+\beta'}>\mu$. Here we see the reason for the constraint $\delta_0\leq \nu/3$ (rather than the seemingly more natural $\delta_0\leq\nu$), since \eqref{toiterate} requires that $\widetilde{U}_k\subset U_\nu(0)$, which is indeed the case for sufficiently small $\delta_0$ independent of $k$.

Choosing $k_*$ to be the smallest natural number $k$ for which $\delta_k^{\alpha+\beta'}\leq\mu$, we may iterate \eqref{toiterate} $k_*$ times to obtain
$$
\mathcal{C}(0,\delta_0,\mu, 1+\varepsilon) \leq \prod_{k=0}^{k_*-1}(1 + \delta_k^{\beta}) \max_{u\in \widetilde{U}_{k_*}} \mathcal{C}\Bigl(u,\delta_{k_*},\mu,(1+\varepsilon)\prod_{k=0}^{k_*-1}\exp(\delta_k^{\beta})\Bigr).
$$
Since $\delta_{k_*}^{\alpha+\beta'}\leq\mu$ we may now apply Proposition \ref{prop:basecase} to obtain
\begin{equation} \label{almostthere}
\mathcal{C}(0,\delta_0,\mu, 1+\varepsilon) \leq (1+\varepsilon)^\sigma \max_{u\in \widetilde{U}_{k_*}} \mbox{BL}(\mathrm{d} \textbf{\mbox{B}}(u), \textbf{\mbox{p}}) \prod_{k=0}^{k_*-1}\left(1 + \delta_k^{\beta}\right)\exp(\sigma\delta_k^{\beta}) .
\end{equation}
At this point, it is clear that we need the Brascamp--Lieb constant to behave sufficiently nicely with respect to the linear mappings; for simple data, it was shown in \cite{ValdBestBest} that the constant is smooth with respect to the linear mappings and we record this fact here (for general data, the constant is continuous which was recently shown in \cite{BBCF} and \cite{GGOW}).
\begin{prop}[\cite{ValdBestBest}] \label{prop:BLcts}
For each fixed $\mathbf{p}$, the function $\BL(\cdot,\mathbf{p})$ is differentiable on the open set $\{\mathbf{L}: (\mathbf{L},\mathbf{p}) \mbox{ simple}\}$.
\end{prop}
By Proposition \ref{prop:BLcts}, we know that $\mbox{BL}(\mathrm{d} \textbf{\mbox{B}}(u), \textbf{\mbox{p}})\leq (1+\varepsilon)\mbox{BL}(\mathrm{d} \textbf{\mbox{B}}(0), \textbf{\mbox{p}})$ for an appropriate choice of $\delta_0 \simeq_\varepsilon 1$.


For the product term in \eqref{almostthere}, note that
\begin{equation*}
\log\left( \prod_{k=0}^{k_*-1}(1+\delta_k^\beta)\exp(\sigma\delta_k^{\beta})\right) \leq (1+\sigma)\sum_{k=0}^{\infty} \delta_k^{\beta} = (1+\sigma)\sum_{k=0}^{\infty} \delta_0^{\alpha^k\beta}\lesssim \delta_0^{\beta}
\end{equation*}
and therefore, for an appropriate choice of $\delta_0 \simeq_\varepsilon 1$, we have $\prod_{k=0}^{k_*-1}(1+\delta_k^\beta)\exp(\sigma\delta_k^{\beta}) \leq 1 + \varepsilon$.
Putting together the above, we have shown
\[
\mathcal{C}(0,\delta_0,\mu, 1+\varepsilon) \leq (1+\varepsilon)^{\sigma + 3} \mbox{BL}(\mathrm{d} \textbf{\mbox{B}}(0), \textbf{\mbox{p}}),
\]
which establishes \eqref{nuf}, and hence Theorem \ref{mainthm}.
\end{proof}

It remains to prove Proposition \ref{prop:Csimple}.

\section{Gaussian preliminaries}\label{sect:gaussians}
Before embarking on our proof of Proposition \ref{prop:Csimple}, we  introduce a family of $m$-tuples of gaussians that will be pivotal in our argument.

As we remark in the introduction, for each $\mathbf{p}$, the set $\{\mathbf{L}:(\mathbf{L},\mathbf{p})\mbox{ simple}\}$ is \emph{open}, and the simplicity of $(\mathbf{L},\mathbf{p})$ implies the existence of a gaussian extremising input $\mathbf{f}$; that is,
an $m$-tuple of functions $(f_j)$ of the form \eqref{gau} for which $$\mbox{BL}(\mathbf{L},\mathbf{p};\mathbf{f})=\mbox{BL}(\mathbf{L},\mathbf{p})$$ -- see \cite{ValdBestBest} and \cite{BCCT1} respectively.
As is also shown in \cite{BCCT1}, this gaussian extremiser $\mathbf{f}$ is \emph{unique} up to the elementary scale-invariance of the Brascamp--Lieb inequality. In \cite{ValdBestBest} Valdimarsson also shows that the map $\mathbf{L}\mapsto\mathbf{A}$ is \emph{smooth} at simple data points, provided the $m$-tuple of positive-definite matrices $\mathbf{A}=(A_j)_{j=1}^m$ is suitably normalised. As a consequence it follows that the map $\mbox{BL}(\cdot,\mathbf{p})$ is also smooth at simple data points -- a fact that we have already used in Section \ref{sect:settingup}.

These results have a clear interpretation in our nonlinear context. In particular, the simplicity hypothesis on $(\mbox{d}\mathbf{B}(0),\mathbf{p})$ in the statement of Theorem \ref{mainthm} ensures the simplicity of $(\mbox{d}\mathbf{B}(u),\mathbf{p})$ provided $u\in\mathbb{R}^n$ is sufficiently small. Furthermore, for each such $u$,
it follows that there is a unique (up to scaling) family of gaussians $\mathbf{g}_u=(g_{u,j})_{j=1}^m$ for which
\begin{equation}\label{brian}
\mbox{BL}(\mbox{d}\mathbf{B}(u),\mathbf{p};\mathbf{g}_u)=\mbox{BL}(\mbox{d}\mathbf{B}(u),\mathbf{p}).
\end{equation}
Here
\begin{equation}\label{steve}
g_{u,j}(x)=c_{u,j}e^{-\pi\langle A_{u,j}x,x\rangle},
\end{equation}
where $A_{u,j}$ is a positive-definite $n_j\times n_j$ matrix, and $c_{u,j}$ a positive real number. Finally, the family $\mathbf{A}_u=(A_{u,j})_{j=1}^m$ may be chosen normalised in such a way that it varies smoothly with $u$. Of course, this also implies that the map $\mbox{BL}(\mbox{d}\mathbf{B}(\cdot),\mathbf{p})$ is smooth in a neighbourhood of the origin -- again, something that we have already used in Section \ref{sect:settingup}.

As in the proof \eqref{Ballineq} of Ball's inequality, it will also be convenient to normalise the functions $g_{u,j}$ so that
$\int g_{j,u}=1$
for each $u$ and $1\leq j\leq m$. This simply amounts to setting $c_{u,j}=\det(A_{u,j})^{1/2}$ in \eqref{steve}.

In what follows the $n\times n$ matrix
$$
M_u := \sum_{j=1}^m p_j (dB_j(u))^* A_{u,j} dB_j(u)
$$
will play a crucial role. It naturally arises via the identity
\begin{eqnarray}\label{maingaus}
\begin{aligned}
\prod_{j=1}^m g_{u,j}(dB_j(u)x)^{p_j}&=\prod_{j=1}^m (\det(A_{u,j}))^{p_j/2}e^{-\pi\langle M_u x,x\rangle}\\&=\mbox{BL}(\mbox{d}\mathbf{B}(u),\mathbf{p})(\det(M_u))^{1/2}e^{-\pi\langle M_u x,x\rangle},
\end{aligned}
\end{eqnarray}
which is a simple consequence of \eqref{brian}.

Since each $A_{u,j}$ is positive definite, each $p_j>0$ and $\mbox{d}\mathbf{B}(u)$ is an $m$-tuple of surjective linear maps satisfying $$\bigcap_{i=1}^m \ker{dB_j(u)}={0},$$ the matrix $M_u$ is positive definite (at least for $u$ sufficiently small). By a further scaling of the $\mathbf{A}_u$ we may assume in addition that
$\det (M_u) = 1.$ This we do merely for convenience. As we shall see, the family of gaussians
$$
e^{-\pi\langle M_u x,x\rangle}
$$
(after suitable truncations) will be our main tool in passing from scale $\delta^\alpha$ to scale $\delta$ in Proposition \ref{prop:Csimple}. To achieve the required localisation at multiple scales, we will need to make use of scalings of $\mathbf{g}_u$, denoted $\mathbf{g}_{u,\rho}$, where
\[ g_{u,\rho,j}(x)=\rho^{-n_j}g_{u,j}(x/\rho).\]
Here $\rho>0$, and we have normalised so that $\int g_{u,\rho,j}=1$.


It will be convenient to introduce certain constants related to $\mathbf{A}_u$. The first is
\begin{equation}\label{eq:barc0}
\bar C_0 := \sup_{u}\|M_u^{-1/2}\|,
\end{equation}
where the supremum is taken over $u$ belonging to some suitably small (compact) neighbourhood of the origin in $\mathbb{R}^n$.
As $M_u$ is positive definite and varies continuously in $u$, we have $\bar C_0 \simeq 1$. The second constant that arises is
\begin{equation}\label{eq:c0}
C_0 = \sup_{u}\max_{j=1,\ldots,m} \|A_{u,j}^{1/2}\|.
\end{equation}
By similar regularity considerations, we also have $C_0 \simeq 1$.


We conclude this section with three technical lemmas concerning the above gaussian extremisers which feed into the proof of Proposition \ref{prop:Csimple} in the next section.


The first of the lemmas states that we can truncate the product of our gaussian extremisers to a suitable compact set without losing too much. This is tantamount to observing that gaussian extremisers become suitable near-extremisers if truncated appropriately.
\begin{lemma}\label{lem:truncation}
For each $\eta>0$ there exists $\nu>0$ such that 
\begin{align*}
\int_{\R^n}\prod_{j=1}^m g_{u,\delta,j}^{p_j}\circ dB_{j}(u)
\leq (1+\delta^{\eta}) \int_{U_\delta(0)}\prod_{j=1}^m g_{u,\delta,j}^{p_j}\circ  dB_{j}(u),
\end{align*}
for all $u\in U_\nu(0)$ and $0<\delta\leq\nu$.
\end{lemma}
\begin{proof}
Fix $\eta>0$. By elementary considerations it suffices to show that
\begin{equation} \label{e:chopgaus}
\int_{\R^n\backslash U_\delta(0)}\prod_{j=1}^m g_{u,\delta,j}^{p_j}\circ  dB_{j}(u)
\lesssim \delta^{2\eta}\int_{\R^n}\prod_{j=1}^m g_{u,\delta,j}^{p_j}\circ dB_{j}(u)
\end{equation}
for sufficiently small $\delta > 0$, depending on $\eta$. To establish \eqref{e:chopgaus} we first observe that the integrand is expressible as
\[
\prod_{j=1}^m g_{u,\delta,j}^{p_j}\circ  dB_{j}(u) = \delta^{-n} \mathrm{BL}(\mathbf{B},\mathbf{p}) e^{-\pi\left\langle M_u \frac{x}{\delta},\frac{x}{\delta}\right\rangle},
\]
which is easily verified using \eqref{maingaus}, the scaling condition \eqref{scal}, and our normalisation $\det M_u =1$. It is clear that \eqref{e:chopgaus} is equivalent to
$$
\int_{| M_u^{-1/2}x| \geq \log(1/\delta)}e^{-\pi|x|^2}dx \lesssim \delta^{2\eta}
$$
and, since we have $| M_u^{-1/2}x| \leq \bar C_0 |x|$ (for sufficiently small $u$), it suffices to show that
\begin{equation} \label{e:chopgausETS}
\int_{|x|\gtrsim \log(1/\delta)}e^{-\pi|x|^2}dx\lesssim \delta^{2\eta}
\end{equation}
for sufficiently small $\delta$.
To see this, we simply observe that
$$
\int_{|x|\gtrsim \log(1/\delta)}e^{-\pi|x|^2}dx\lesssim \int_{|x|\sim\log(1/\delta)}e^{-\pi|x|^2}dx\lesssim (\log(1/\delta))^ne^{-c(\log(1/\delta))^2}=(\log(1/\delta))^n\delta^{c\log(1/\delta)}
$$
provided $\delta$ is sufficiently small and $c \simeq 1$.
\end{proof}

The next technical lemma we require is the following, which morally says that gaussian extremisers are stable under perturbations of their centres by an amount which is small relative to their scale. Before providing the precise statement, we introduce the affine mappings $L_j^{u,y} : \mathbb{R}^n \to \mathbb{R}^{n_j}$ given by
\[
L^{u,y}_jx = L^u_jx - B_j(y) = B_j(u)+dB_j(u)(x-u) - B_j(y).
\]
\begin{lemma}\label{lem:perturb} There exists $\nu \simeq 1$ such that if $0<\delta<\nu$, then for $u\in U_{\nu}(0)$ and $y\in U_\delta(u)$,
\begin{align*}
\int_{U_{\delta^\alpha}(y)}\prod_{j=1}^m (g_{u,\delta^\alpha,j}^{p_j} \circ dB_j(u))(x-y)dx \leq (1+\delta^{\beta'}) \int_{U_{\delta^\alpha}(y)}\prod_{j=1}^m g_{u,\delta^\alpha,j}^{p_j} \circ L^{u,y}_j.
\end{align*}
\end{lemma}

\begin{proof} Let $\nu\simeq 1$ be a small constant to be chosen below. Fix $\delta$  such that $0 < \delta < \nu$, $u\in U_{\nu}(0)$, and $y\in U_\delta(u)$. It then suffices to prove, for any $\gamma$ satisfying $\beta' < \gamma < 2 - \alpha$, that
\begin{equation}\label{beforebootstraping}
\int_{U_{\delta^\alpha}(y)}\prod_{j=1}^m (g_{u,\delta^\alpha,j}^{p_j} \circ dB_j(u))(x-y)dx \leq \delta^{\gamma} + \int_{U_{\delta^\alpha}(y)}\prod_{j=1}^m g_{u,\delta^\alpha,j}^{p_j} \circ L^{u,y}_j.
\end{equation}
as we may then obtain the desired estimate using a bootstrapping argument. 
Indeed, by Proposition \ref{prop:BLcts} and Lemma \ref{lem:truncation},
\[
{\mbox{BL}(\mathrm{d}{\bf B}(0), {\bf p})}\leq 2{\mbox{BL}(\mathrm{d}{\bf B}(u), {\bf p})}
 \leq 4\int_{U_{\delta^\alpha}(y)}\prod_{j=1}^m (g_{u,\delta^\alpha,j}^{p_j} \circ dB_j(u))(x-y)dx
\]
for $u\in U_{\nu}(0)$, provided $\nu$ is sufficiently small. Therefore \eqref{beforebootstraping} implies
\[
\left(1-\frac{4\delta^{\gamma}}{{\mbox{BL}(\mathrm{d}{\bf B}(0), {\bf p})}}\right)
\int_{U_{\delta^\alpha}(y)}\prod_{j=1}^m (g_{u,\delta^\alpha,j}^{p_j} \circ dB_j(u))(x-y)dx  \nonumber\\
\leq \int_{U_{\delta^\alpha}(y)}\prod_{j=1}^m g_{u,\delta^\alpha,j}^{p_j} \circ L^{u,y}_j,
\]
and the result quickly follows, provided $\nu \simeq 1$ is sufficiently small.

In order to show \eqref{beforebootstraping}, we write $\ell = \sum_{j=1}^m n_j$ and, identifying $\mathbb{R}^{n_1} \times \cdots \times \mathbb{R}^{n_m}$ with $\mathbb{R}^\ell$, we define
\[
G(w) = \prod_{j=1}^m g_{u,\delta^\alpha,j}(w_j)^{p_j}
\]
for $w \in \mathbb{R}^\ell$. Here, and for the remainder of the proof, we suppress the dependence on $u$ and $y$. Thus, \eqref{beforebootstraping} follows once we show
\begin{equation} \label{e:perturbETS}
\|G \circ L -  G \circ \widetilde{L}\|_{L^1(U_{\delta^\alpha}(y))} \leq \delta^{\gamma},
\end{equation}
where $L_j = L^{u,y}_j$ and $\widetilde{L}_j = dB_j(u)(\cdot-y)$, and we shall obtain this using an appropriate uniform bound on the integrand. For this, first we observe the estimate
\[
|\nabla G(w)| \lesssim \delta^{-\alpha(n+2)} |w|.
\]
This follows from $\prod_{j=1}^m c_{u,j}^{p_j} \lesssim 1$, the fact that  $C_0 \simeq 1$, where $C_0$ is given by \eqref{eq:c0}, and using the scaling condition \eqref{scal} to collect powers of $\delta$. Since we also have
\[
|L_jx - \widetilde{L}_jx| \lesssim \delta^2 \left(\log \frac{1}{\delta}\right)^2
\]
from \eqref{taylor}, the mean value theorem implies
\[
\|G \circ L -  G \circ \widetilde{L}\|_{L^\infty(U_{\delta^\alpha}(y))} \lesssim \delta^{-\alpha(n+2)} \delta^{\alpha +2} \left(\log \frac{1}{\delta}\right)^3
\]
and hence
\[
\|G \circ L -  G \circ \widetilde{L}\|_{L^1(U_{\delta^\alpha}(y))}  \lesssim \delta^{2-\alpha}  \left(\log \frac{1}{\delta}\right)^{n+3}.
\]
Clearly we obtain \eqref{e:perturbETS}, provided $\nu \simeq 1$ is sufficiently small, and this completes the proof.
\end{proof}

The final lemma is concerned with ``local constancy" of gaussians. Whilst gaussians are not $\kappa$-constant at scale $\mu$ on $\R^n$ for any choice of $\kappa$ and $\mu$, there exists $\kappa$ such that they are $\kappa$-constant at scale $\mu$ on the truncated neighbourhoods under consideration.
\begin{lemma}\label{lem:gaussconstant}
There exists $\nu \simeq 1$ such that the following holds: for all $\delta \in (0,\nu)$ and $\mu>0$ satisfying the threshold condition $\delta^{\alpha+\beta'}>\mu$, and for all $u \in U_{\nu}(0)$ and $x\in 2U_\delta(u)$, the function $g_{u,\delta^\alpha,j}(L^u_jx-\cdot)$ is $\exp(\delta^\beta)$-constant at scale $\mu$ on $B_j(2U_{\delta^\alpha}(x))$.
\end{lemma}
\begin{proof}
It suffices to show that, given any $C \simeq 1$, the function $g_{u,\delta^\alpha,j}$ is $\exp(\delta^\beta)$-constant at scale $\mu$ on $CU_{\delta^\alpha}(0)$, where $u \in U_{\nu}(0)$, for an appropriate choice of $\nu \simeq 1$, and $\delta \in (0,\nu)$ and $\mu>0$ satisfing the threshold condition $\delta^{\alpha+\beta'}>\mu$. Indeed, $g_{u,\delta^\alpha,j}(L^u_jx-\cdot)$ is simply a translation of $g_{u,\delta^\alpha,j}$, so to obtain the desired claim, it is only necessary to check that $L_j^ux - w \in CU_{\delta^\alpha}(0)$ for some $C \simeq 1$ whenever $w \in B_j(2U_{\delta^\alpha}(x))$. This follows since, given such $w \in B_j(2U_{\delta^\alpha}(x))$ and writing $w = B_j(y)$ for some $y \in  2U_{\delta^\alpha}(x)$, we may use \eqref{taylor} to obtain
\[
|L^u_jx - w| \lesssim |x-u|^2 + |B_j(x) - B_j(y)| \lesssim \delta^2 \left(\log \frac{1}{\delta}\right)^2 + \delta^\alpha \log \frac{1}{\delta^\alpha} \sim\delta^\alpha \log \frac{1}{\delta^\alpha}.
\]
To establish that $g_{u,\delta^\alpha,j}$ is $\exp(\delta^\beta)$-constant at scale $\mu$ on $CU_{\delta^\alpha}(0)$, we first observe that, for any $R \geq 1$, $\widetilde{\mu} \leq R$, $|x| \leq R$ and $|y| \leq \widetilde{\mu}$, we have
\[
g_{u,j}(x) \leq \exp(3\pi R\|A_{u,j}\|\widetilde{\mu}) g_{u,j}(x+y)
\]
by elementary considerations. Applying this fact with $R = C\alpha\log(1/\delta)$ and $\widetilde{\mu} = \delta^{-\alpha}\mu$, we obtain that $g_{u,\delta^\alpha,j}$ is $\exp(\eta)$-constant at scale $\mu$ on $CU_{\delta^\alpha}(0)$, where $\eta = 3\pi C\alpha\|A_{u,j}\|\mu\delta^{-\alpha}\log(1/\delta)$. In order to use the preceding observation, we should ensure $\widetilde{\mu} \leq R$; this is simply a consequence of the fact that $\beta'$ is positive (and provided $\nu$ is sufficiently small). Moreover, under the threshold condition we have $\delta^{\alpha+\beta'}>\mu$, and for $u \in U_{\nu}(0)$, we have
$\eta \lesssim \delta^{\beta'}\log(1/\delta)\lesssim \delta^\beta\nu^{(\beta'-\beta)/2}$.
By making $\nu \simeq 1$ small enough,
 it follows that $\eta \leq \delta^\beta$. This establishes that $g_{u,\delta^\alpha,j}$ is $\exp(\delta^\beta)$-constant at scale $\mu$ on $CU_{\delta^\alpha}(0)$, for any $u \in U_{\nu}(0)$, and thus completes the proof of the lemma.
\end{proof}


\section{The proof of Proposition \ref{prop:Csimple}}\label{sect:step}
The main idea in our proof of Proposition \ref{prop:Csimple} is that it is possible to run a certain nonlinear variant of the argument leading to \eqref{Ball2}. In doing so, we first claim that it is enough to prove that there exists $\nu \simeq 1$ such that if $u\in U_{\nu}(0)$ and $\delta \in (0,\nu)$ satisfies $\delta^{\alpha+\beta'}>\mu$, then
\begin{equation}\label{inductionstep'}
\mathcal{C}(u,\delta,\mu,\kappa) \leq (1 + \delta^{\beta'})^2 \max_{x\in 2U_\delta(u)} \mathcal{C}(x,\delta^\alpha,\mu,\kappa\exp(\delta^\beta)).
\end{equation}
Clearly the desired inequality in \eqref{inductionstep} follows from \eqref{inductionstep'} by choosing $\nu \simeq 1$ sufficiently small. Our proof appeals to Lemmas \ref{lem:truncation}--\ref{lem:gaussconstant} in sequence, with each implicitly imposing its own smallness requirement on $\nu$.
\begin{proof}[Proof of \eqref{inductionstep'}]
We fix $u \in U_{\nu}(0)$, and assume $\delta \in (0,\nu)$ and $\mu$ satisfy the threshold condition $\delta^{\alpha+\beta'}>\mu$. We also fix $f_j\in L^1(B_j(2U_\delta(u));\mu,\kappa)$ and, to shorten some forthcoming expressions, we write $F = \prod_{j=1}^m (f_j \circ B_j)^{p_j}$.

By \eqref{brian}, and elementary scaling, we have
\begin{equation}\label{brianscal}
\mbox{BL}(\mbox{d}\mathbf{B}(u),\mathbf{p};\mathbf{g}_{u,\delta^\alpha})=\mbox{BL}(\mbox{d}\mathbf{B}(u),\mathbf{p}),
\end{equation}
and so
$$
\int_{U_\delta(u)} F
= \frac{1}{\mbox{BL}(\mathrm{d}{\bf B}(u), {\bf p})} \int_{U_\delta(u)} F(y) \int_{\R^n}\prod_{j=1}^m g_{u,\delta^\alpha,j}^{p_j}(dB_j(u)x) dxdy.
$$
By Lemma \ref{lem:truncation} with $\eta=\beta'/\alpha$, followed by a translation change of variables in $x$, we obtain
\begin{equation*}
\int_{U_\delta(u)} F
\leq \frac{1+\delta^{\beta'}}{\mbox{BL}(\mathrm{d}{\bf B}(u), {\bf p})} \int_{U_\delta(u)}F(y) \int_{U_{\delta^\alpha}(y)}\prod_{j=1}^m g_{u,\delta^\alpha,j}^{p_j}(dB_j(u)(x-y))  dxdy.
\end{equation*}
Now we may apply Lemma \ref{lem:perturb}, and then Fubini's theorem, to obtain
\begin{align*}
\int_{U_\delta(u)}F & \leq \frac{(1+ \delta^{\beta'})^2}{\mbox{BL}(\mathrm{d}{\bf B}(u), {\bf p})}  \int_{ U_\delta(u)}\int_{U_{\delta^\alpha}(y)}F(y) \prod_{j=1}^m g_{u,\delta^\alpha,j}(L^{u,y}_jx)^{p_j}dxdy \\
& \leq \frac{(1+ \delta^{\beta'})^2}{\mbox{BL}(\mathrm{d}{\bf B}(u), {\bf p})}  \int_{U_{\delta}(u)+U_{\delta^\alpha}(0)} \int_{U_{\delta^\alpha}(x)} F(y) \prod_{j=1}^m g_{u,\delta^\alpha,j}(L^{u,y}_jx)^{p_j}dydx.
\end{align*}
Here, we recall that $L^{u,y}_jx$ is the affine map given by
\[
L^{u,y}_jx = L^u_jx - B_j(y) = B_j(u)+dB_j(u)(x-u) - B_j(y).
\]
The inner integral is thus,
\[
\int_{U_{\delta^\alpha}(x)}\prod_{j=1}^m \left(h_j^x\circ B_j\right)^{p_j}
\]
where
\[
h_j^x(w) = f_j(w)g_{u,\delta^\alpha,j}(L^u_jx - w).
\]
By assumption, $f_j\in L^1(B_j(2U_\delta(u));\mu,\kappa)$ and, for $x\in U_\delta(u)+U_{\delta^\alpha}(0)$, we have
$2U_{\delta^\alpha}(x) \subset 2U_\delta(u)$. Hence, $f_j\in L^1(B_j(2U_{\delta^\alpha}(x));\mu,\kappa).$  Moreover, according to Lemma \ref{lem:gaussconstant}, we have $$g_{u,\delta^\alpha,j}(L^u_jx-\cdot)\in L^1(B_j(2U_{\delta^\alpha}(x));\mu,\exp(\delta^\beta)),$$ and therefore
\[
h_j^x\in L^1(B_j(2U_{\delta^\alpha}(x));\mu,\kappa\exp(\delta^\beta))
\]
whenever $x\in U_\delta(u)+U_{\delta^\alpha}(0)$.

By the definition of $\mathcal{C}(x,\delta^\alpha,\mu,\kappa\exp(\delta^\beta))$, and slightly enlarging the domain of $x$,
\begin{align*}
\int_{U_\delta(u)}F  & \leq \frac{(1+ \delta^{\beta'})^2\max_{x \in 2U_{\delta}(u)} \mathcal{C}(x,\delta^\alpha,\mu,\kappa\exp(\delta^\beta))}{\mbox{BL}(\mathrm{d}{\bf B}(u), {\bf p})}  \int_{2U_{\delta}(u)} \prod_{j=1}^m \left(\int_{\R^{n_j}} h_j^x\right)^{p_j}dx \\
& = \frac{(1+ \delta^{\beta'})^2\max_{x \in 2U_{\delta}(u)} \mathcal{C}(x,\delta^\alpha,\mu,\kappa\exp(\delta^\beta))}{\mbox{BL}(\mathrm{d}{\bf B}(u), {\bf p})}  \int_{2U_{\delta}(u)} \prod_{j=1}^m (f_j*g_{u,\delta^\alpha,j})^{p_j} \circ L^u_j.
\end{align*}
To conclude, we simply apply the linear inequality and the fact that the gaussians $g_{u,\delta^\alpha,j}$ are normalised in $L^1$ to obtain
\begin{align*}
\int_{U_\delta(u)} F & \leq (1+ \delta^{\beta'})^2\max_{x \in 2U_{\delta}(u)} \mathcal{C}(x,\delta^\alpha,\mu,\kappa\exp(\delta^\beta)) \prod_{j=1}^m \left(\int_{\R^{n_j}}f_j*g_{u,\delta^\alpha,j}(x+B_j(u))dx\right)^{p_j} \\
& =  (1+ \delta^{\beta'})^2 \max_{x\in 2U_\delta(u)} \mathcal{C}(x,\delta^\alpha,\mu,\kappa\exp(\delta^\beta))\prod_{j=1}^m \left(\int_{\R^{n_j}}f_j\right)^{p_j}
\end{align*}
and hence \eqref{inductionstep'}.
\end{proof}

\end{document}